\documentclass[11pt,oneside,a4paper,reqno]{amsart}

 \usepackage[stable,symbol]{footmisc}

\usepackage{hyperref}
\usepackage{bm}
\usepackage{amsmath}
\usepackage{amsthm}
\usepackage{amssymb}

\newcommand{\mathscr}{\mathcal}

\newcommand{\mn}{\mathrm{min}}

\newcommand{\half}{\frac{1}{2}}
\newcommand{\calc}{\mathcal{C}}

\newcommand{\rr}{{\mathbb R}}
\newcommand{\cc}{{\mathbb C}}

\numberwithin{equation}{section}

\DeclareMathOperator{\dom}{dom}
\DeclareMathOperator{\spec}{spec}

\DeclareMathOperator{\ran}{ran}
\DeclareMathOperator{\pv}{p.v.}

\newcommand{\spece}{\spec_\mathrm{ess}}

\newcommand{\cD}{{\mathcal D}}
\newcommand{\cF}{{\mathcal F}}
\newcommand{\cS}{{\mathcal S}}

\newcommand{\eps}{\varepsilon}

\newtheorem{theorem}{Theorem}[section]

\newtheorem{lemma}[theorem]{Lemma}
\newtheorem{corollary}[theorem]{Corollary}

\theoremstyle{definition}

\newtheorem{remark}[theorem]{Remark}

\title[]{On the self-adjointness of two-dimensional relativistic shell interactions}

\author{Badreddine Benhellal, Konstantin Pankrashkin, Mahdi Zreik}

\address{(B. Benhellal) National Higher School of Mathematics, Scientific and Technology Hub of Sidi Abdellah, P.O. Box 75,  16093 Algiers, Algeria}

\email{badreddine.benhellal@nhsm.edu.dz}

\address{(K. Pankrashkin) Carl von Ossietzky Universit\"at Oldenburg, Fakultät V, Institut f\"ur Mathematik, 26111 Oldenburg, Germany}

\email{konstantin.pankrashkin@uol.de}

\address{(M. Zreik) Laboratoire de Math\'ematiques d'Orsay, Universit\'e Paris-Saclay, 91405 Orsay Cedex, France}
\email{mahdi.zreik@universite-paris-saclay.fr}



\begin{document}
	\begin{abstract}  We discuss the self-adjointness of the two-dimensional Dirac operator with a transmission condition along a closed Lipschitz curve. The main new ingredients are an explicit use of the Cauchy transform on non-smooth curves and a direct link with the Fredholmness of a singular boundary integral operator.
	The new self-adjointness condition includes and extends all previous results for this class of problems. The study is particularly precise for the case of curvilinear polygons, as the angles can be taken into account in an explicit way.
	
	\bigskip
	
	\it 
	This is a \underline{postprint} version. The paper already appeared in J. Operator Theory {\bfseries 93} (2025) 569--591; however, we discovered a serious mistake in Section 5 after the publication (an official erratum will appear soon in the journal). The present text contains a corrected version of Section~5 and an adapted introduction.	
	\end{abstract}
	
\maketitle	

\tableofcontents

\section{Introduction}

Dirac operators with $\delta$-interactions supported on general hypersurfaces have been actively studied since the appearance of the paper \cite{AMV1}. Due to the presence of distributional coefficients, the self-adjointness of such operators requires special attention, and it was seen by many authors (primarily for the three-dimensional case) that the self-adjointness domain can be dependent on the coupling constants and the smoothness properties of the hypersurface and that it may lead to unusual spectral properties \cite{b1,b2,BB2,BB3,BP}. The paper \cite{BHOP} initiated the study of the two-dimensional case, and for the case of smooth curves a very complete spectral picture could be found, which was extended in \cite{CLMT} to a more general class of interactions. Much less attention was given to the case of non-smooth surfaces and curves. In the present work, we discuss the self-adjointness of two-dimensional Dirac operators with  $\delta$-interactions supported on closed Lipschitz curves (in particular, on curvilinear polygons). Our results complement those obtained in the recent papers \cite{BHSS,PV} and provide precise ranges of coupling constants and corner openings for which the domain of self-adjointness can be given explicitly. Compared to the preceding works, we employ two new technical ingredients: the explicit use of the Cauchy transform on non-smooth curves and a characterization of the Fredholmness for boundary integral operators.

Now let us pass to precise formulations. Throughout the text, we use the Pauli matrices
\[
\sigma_{1}=\left(\begin{array}{cc}
0 & 1 \\
1 & 0
\end{array}\right) ,\quad \sigma_{2}=\left(\begin{array}{cc}
0 & -i \\
i & 0
\end{array}\right) ,\quad \sigma_{3}=\left(\begin{array}{cc}
1& 0 \\
0 & -1
\end{array}\right)
\]
and denote by $\sigma_0$ the $2\times 2$ identity matrix. The anticommutation relations
\begin{align}\label{anticmu}
    \lbrace \sigma_{j},\sigma_{k}\rbrace = \sigma_{j}\sigma_{k}+\sigma_{k}\sigma_{j}=2\delta_{jk}\sigma_0 \text{ for all }j,k\in \lbrace 1,2,3\rbrace
\end{align}
are well known. Let $m\in\rr$. The two-dimensional Dirac operator with mass $m$  is the formally self-adjoint differential expression
\[
D:\ C^\infty_0(\rr^2,\cc^2)\ni f\ \mapsto\  -i(\sigma_1\partial_1 f+\sigma_2\partial_2 f)+m\sigma_3 f\in C^\infty_0(\rr^2,\cc^2),
\]
and it naturally extends to a continuous linear map in the space of distributions $\cD'(\Omega,\cc^2)$ for any open set $\Omega\subset\rr^2$.
It is well known that the operator
\begin{equation}
	\label{adom}
A:\  f\mapsto Df, \quad \dom A=H^1(\rr^2,\cc^2)
\end{equation}
(the free two-dimensional Dirac operator),  is self-adjoint in $L^2(\rr^2,\cc^2)$ and has the absolutely continuous spectrum
\[
\spec A=\big(-\infty,-|m|\big]\,\cup\big[|m|,+\infty\big),
\]
and it occupies a central place in relativistic quantum mechanics \cite{Tha}. We will be interested in the study of some special perturbations of $A$.

Namely, let $\Omega_+\subset\rr^2$ be a non-empty bounded open set with Lipschitz boundary. Denote
\[
\Sigma:=\partial\Omega_+,\quad \Omega_-:=\rr^2\setminus\overline{\Omega_+}.
\]
For $(\eps,\mu)\in\rr^2$ we would like to discuss self-adjoint realizations in $L^2(\rr^2,\cc^2)$ 
of operators given formally by
\begin{equation}
	\label{dirac2}
f\mapsto Df+(\eps\sigma_0 + \mu \sigma_3)\delta_\Sigma f,
\end{equation}
where $\delta_\Sigma$ is the Dirac $\delta$-distribution supported on $\Sigma$. The last summand can be considered as an idealized model of a relativistic potential concentrated on $\Sigma$, and the constant $\eps$ resp. $\mu$ measures the strength of the electrostatic resp. Lorentz scalar part of the interaction.
The formal expression \eqref{dirac2} can be given a more rigorous meaning as follows.
 First, for any non-empty open set $\Omega\subset\rr^2$ consider the space
\[
H(\sigma,\Omega):=\Big\{f\in L^2(\Omega,\cc^2):\ Df\in L^2(\Omega,\cc^2)
\Big\},
\]	
which is just the domain of the maximal realization of $D$ in $L^2(\Omega,\cc^2)$ and becomes a Hilbert space if equipped with the scalar product
\[
\langle f,g\rangle_{H(\sigma,\Omega)}:=\langle f,g\rangle_{L^2(\Omega,\cc^2)}+\langle Df,Dg\rangle_{L^2(\Omega,\cc^2)}.
\]
For $s>0$ let $H^s(\Omega,\cc^2)$ be the usual fractional Sobolev spaces of order $s$ on $\Omega$ (consisting of $\cc^2$-valued functions), and we set
\[
H^s(\sigma,\Omega):=H(\sigma,\Omega)\cap H^s(\Omega,\cc^2),
\]
which is a Hilbert space with the scalar product
\[
\langle f,g\rangle_{H^s(\sigma,\Omega)}:=\langle f,g\rangle_{H(\sigma,\Omega)}+\langle f,g\rangle_{H^s(\Omega,\cc^2)}.
\]

For what follows it will be convenient to use the identification
\[
H(\sigma,\rr^2\setminus\Sigma)\simeq H(\sigma,\Omega_+)\oplus H(\sigma,\Omega_-),
\quad
f\simeq (f_+,f_-),
\]
with $f_\pm$ being the restriction of $f$ on $\Omega_\pm$, as well as the analogous identifications for $H^s(\rr^2\setminus\Sigma,\cc^2)$
and $H^s(\sigma,\rr^2\setminus\Sigma)$. We will also use the shorthand notation
\[
\sigma\cdot x:=x_1\sigma_1+x_2\sigma_2,\quad x=(x_1,x_2)\in\rr^2;
\]
from the anticommutation relations \eqref{anticmu} one easily obtains $(\sigma\cdot x)^2=|x|^2\sigma_0$ for all $x\in\rr^2$.

It is known that for any $f\in H(\sigma,\rr^2\setminus\Sigma)$ the boundary traces $(\sigma\cdot \nu)f_\pm$ on $\Sigma$ are well-defined as functions in $H^{-\frac{1}{2}}(\Sigma,\cc^2)$; remark that we keep the same symbols for the boundary traces for better readability. Denote by $\delta_\Sigma f$ the distribution
\[
\langle \delta_\Sigma f, \varphi\rangle :=\int_\Sigma \dfrac{f_+ + f_-}{2}\,\varphi \,\mathrm{d} s, \quad \varphi\in C^\infty_c(\rr^2),
\]
where $\mathrm{d}s$ means the integration with respect to the arclength. An application of the jump formula (distributional derivative for functions with discontinuities along $\Sigma$)
for a function $f$ shows the identity
\begin{align}\label{distjump}
D f= (Df_+)\oplus (Df_-)+i(\sigma\cdot\nu )(f_+-f_-)\delta_\Sigma,
\end{align}
where $\nu=(\nu_1,\nu_2)$ is the unit normal on $\Sigma$ pointing to $\Omega_-$. Then it follows that the right-hand side of \eqref{dirac2} remains a function if and only if the coefficients by $\delta_\Sigma$ are compensated, i.e. if $f$ satisfies the transmission condition
\begin{align}\label{TC}
	 (\eps \sigma_0 + \mu\sigma_{3})\frac{f_+ + f_-}{2}+  i(\sigma\cdot\nu )(f_+-f_-) = 0 \,\text{ on } \Sigma
\end{align}
in a suitable sense.

Therefore, as a first attempt, it is natural to consider the following minimal operator realization $B_\mn$ of the expression \eqref{dirac2} in $L^2(\rr^2,\cc^2)$:
\begin{gather*}
\dom B_\mn:=\big\{f\in H^1(\rr^2\setminus\Sigma,\cc^2):\ f \text{ satisfies \eqref{TC}}\big\},\\
B_\mn:\ f\mapsto Df \text{ in }\rr^2\setminus\Sigma.
\end{gather*}
It is standard to see that $B_\mn$ is symmetric, therefore, $B_\mn\subset B\subset B_\mn^*$ for any self-adjoint realization $B$ of \eqref{dirac2}. Nevertheless, an explicit description of the self-adjoint realizations turns out to be an involved problem depending on both $(\eps,\mu)$ and the regularity of $\Sigma$.

The most attention was given to the case of $C^2$-smooth $\Sigma$, see \cite{BHSS} and references therein.
Namely, if $\eps^2-\mu^2\ne 4$, then $B_\mn$ turns out to be self-adjoint, and its spectrum  consists of the spectrum of the free Dirac operator $A$ and at most finitely many discrete eigenvalues in $(-|m|,|m|)$. For $\eps^2-\mu^2=4$ the operator $B_\mn$ is not closed but is essentially self-adjoint, and the loss of regularity leads to peculiar spectral effects for its unique self-adjoint extension (e.g. new pieces of the essential spectrum), see \cite{BHOP,BHSS,BP}. Remark that \cite{BHSS,CLMT} actually consider more general interactions
by admitting so-called anomalous magnetic couplings which are not covered by the above framework.

If $\Sigma$ has corners, then, in general, $\overline{B_\mn}\subsetneq B_\mn^*$ and $B_\mn$ has infinitely many self-adjoint extensions \cite{OP}. The work \cite{OP} suggested that the $H^\frac{1}{2}$ regularity should be more natural for the case of non-smooth $\Sigma$. Namely, let
\[
B\equiv B_{\eps,\mu}
\]
be defined by
\begin{equation}
	\label{bdom}
	\begin{aligned}
		B&:\ f\simeq(f_+,f_-)\mapsto (D f_+,Df_-),\\
		\dom B&:=\Big\{
		f\in H^\half(\sigma,\rr^2\setminus\Sigma):\ f\text{ satisfies \eqref{TC}}
		\Big\}.
	\end{aligned}	
\end{equation}
Due to \cite[Lem.~4.1]{BHSS} the one-sided traces of functions from $\dom B$ on $\Sigma$ belong to $L^2(\Sigma,\cc^2)$, so the integration by parts shows that $B$ is a symmetric operator. The main result of \cite{PV} reads as follows: if $\Sigma$ is a curvilinear polygon (a piecewise $C^2$-smooth closed curve, with finitely many corners and without cusps), $\eps=0$ and $|\mu|<2$, then $B$ is self-adjoint. The recent work \cite{BHSS}  presents an extensive study of the case of general compact  Lipschitz curves $\Sigma$ by reducing the self-adjointness to the Fredholmness of some boundary integral operator (see also \cite{AMV1, BB3} for the three-dimensional case):
we summarize the essential components of the constructions in Section~\ref{secprep}. Nevertheless, the self-adjointness conditions obtained in \cite{BHSS} for our case are quite implicit as they depend on the (unknown) spectra of some boundary integral operators.

In the present work we extend the results of both \cite{BHSS} and \cite{PV} by providing new very explicit conditions for the self-adjointness of $B$ in terms of the parameters $(\eps,\mu)$ and the geometry of $\Sigma$. Namely, we show that $B$ is self-adjoint in the following cases:
\begin{itemize}
	\item[(A)] The curve $\Sigma$ is Lipschitz and $|\eps|\le |\mu|$ (Corollary \ref{bleq}),
	\item[(B)] The curve $\Sigma$ is $C^1$-smooth and $\eps^2-\mu^2\ne 4$ (Theorem \ref{thmc1}),
	\item[(C)] The curve $\Sigma$ is a curvilinear polygon (with $C^1$-smooth edges and without cusps) and
	\[
	\eps^2-\mu^2<\frac{1}{m (\omega)} \qquad \text{ or } \qquad \eps^2-\mu^2>16 m(\omega),
	\]
	where the constant $m(\omega)\ge\tfrac{1}{4}$ only depends on the sharpest corner $\omega$ of $\Sigma$ (Theorem \ref{thm-3})
	and is related to the essential norm of the Cauchy transform on $\Sigma$ (see Remark~\ref{rem-7}).
	The exact value of $m(\omega)$ is not known explicitly for all $\omega$, but some bounds can be obtained,
	and each of the conditions
	\begin{itemize}
	\item[(i)] $\eps^2-\mu^2\notin\Big(\tfrac{4}{3+2\sqrt{2}},4\big(3+2\sqrt{2}\big)\Big)$,
	\item[(ii)] $\eps^2-\mu^2\notin[2,8]$ and the interior angles $\theta_j$ of $\Sigma$ satisfy
	\[
	\dfrac{\pi}{2}\le \theta_j\le \dfrac{3\pi }{2} \text{ for all } j,
	\]
\end{itemize} 
	guarantees the self-adjointness of $B$ (Corollary \ref{corol-6}).
\end{itemize}

\begin{remark} If the operator $B$ is self-adjoint, a standard analysis shows that its essential spectrum coincides with the spectrum of the free Dirac operator $A$ and that the discrete spectrum is at most finite \cite[Proposition 3.8]{BHOP}. While all constructions of~\cite{BHOP} are formally for smooth $\Sigma$, the proof of this specific result only uses the compact embedding of $H^s(\Omega)$ to $L^2(\Omega)$ for $s>0$ and bounded open sets $\Omega\subset\rr^2$ with Lipschitz boundaries.
\end{remark}

\section{Preparations for the proof}\label{secprep}

We will need some constructions related to the free Dirac operator $A$ in \eqref{adom}. Most of these required results were already obtained in \cite{BHOP,BHSS} and we simply present them
in an adapted form.

First of all, we consider the Cauchy transform on $\Sigma$, i.e. the linear operator $C_\Sigma:\mathit{L}^2(\Sigma)  \longrightarrow \mathit{L}^2(\Sigma)$  defined through the complex line integration
\begin{equation*}
	C_\Sigma g(x):= \dfrac{i}{2\pi}\,\pv\int_{\Sigma}\dfrac{g(y)}{x-y}\,\mathrm{d}y,\quad g\in L^2(\Sigma),\quad x\in \Sigma,
\end{equation*}
and understood in the Cauchy principal value sense. It is a classical result that $C_\Sigma$ is well-defined and bounded \cite{CMM}. Moreover, if one considers the analytic function
\[
F_g:\ \cc\setminus\Sigma\simeq\rr^2\setminus\Sigma\ni x\mapsto \dfrac{i}{2\pi}\,\pv\int_{\Sigma}\dfrac{g(y)}{x-y}\,\mathrm{d}y,\quad g\in L^{2}(\Sigma),
\]
then Plemelj-Sokhotski formulas are valid:
\[
F_g(x)=\pm\frac{g(x)}{2}+C_\Sigma g(x) \text{ for a.e. $x\in\Sigma$},
\]
where the value on the left-hand side is understood as the non-tangential limit  \cite[p.~108]{jj}.

Denote by $K_j$ the modified Bessel functions of order $j$. For $z\in\cc\setminus\spec A$ consider the function $\phi_z:\rr^2\to M_{2\times 2}(\cc)$ given by
\begin{multline*}
    \phi_z (x):= \dfrac{1}{2\pi} K_{0}\big( \sqrt{ m^{2} - z^{2}} |x| \big)\big( m\sigma_{3} + z\sigma_0\big)\\
    {} + i\dfrac{\sqrt{ m^{2} - z^{2}}}{2\pi |x|} K_{1}\big( \sqrt{ m^{2} - z^{2}} |x| \big)( \sigma\cdot x).
\end{multline*}
It will be convenient to admit the additional value $z=m$ by setting
\begin{align*}
     \phi_m (x):=\frac{i}{2\pi}\left(\begin{array}{cc}
0& \dfrac{1}{x_1+ix_2} \\
\dfrac{1}{x_1-ix_2} & 0
\end{array}\right).
\end{align*}
Using the asymptotic expansions of $K_j$ one obtains
\begin{equation} 
	 \label{fifi}
	\phi_z(x) = \phi_m(x) + h_1(x) \log|x| +h_2(x)
\end{equation}
with continuous functions $h_j$, see \cite[Lemma 3.3]{BHOP} for details.

For $z\in\cc\setminus\spec A$ the function $\phi_z$ is a fundamental solution of $D-z$,
and it gives rise to several (singular) integral operators.

Namely, consider the layer potentials $\Phi_z$ for $D-z$ (with $z\in\cc\setminus\spec A$)
\begin{align*}
\begin{split}
\Phi_{z}:&\  L^2(\Sigma,\cc^{2})  \longrightarrow \mathit{L}^2(\rr^2,\cc^2),\\
\Phi_{z}g(x) =&\int_{\Sigma} \phi_{z}(x-y)g(y)\,\mathrm{d}s(y), \quad  x\in\rr^2\setminus\Sigma,
\end{split}
\end{align*} 
where we recall that $\mathrm{d}s$ means the integration with respect to the arclength. Observe that  $\phi_{z}(x)^{\ast}=\phi_{\Bar{z}}(-x)$ for all $x$. Let $\gamma:H^1(\rr^2,\cc^2)\to L^2(\Sigma,\cc^2)$ be the Sobolev trace operator (which is a bounded linear operator), then for any $u \in L^2(\rr^2, \cc^2)$ and $g \in L^2(\Sigma, \cc^2)$ one has, using Fubini's theorem,
\begin{align*}
	\langle \Phi_{\Bar{z}} g, u \rangle_{L^2(\rr^2, \mathbb{C}^2)}&=\int_{\rr^2} \Big\langle \int_\Sigma \phi_{\Bar{z}}(x-y)g(y)\,\mathrm{d}s(y), u(x)\Big\rangle_{\cc^2}\,\mathrm{d}x\\
	&=  \int_\Sigma \Big\langle g(y), \int_{\rr^2} \phi^{\ast}_{\Bar{z}}(x-y)u(x)\mathrm{d} x \Big\rangle_{\cc^2}\,\mathrm{d}s(y),\\
	&= \big\langle g, \gamma (A- z)^{-1}u\big\rangle_{L^2(\Sigma,\cc^2)}.
\end{align*} 
This shows that $\Phi_{\Bar z}=\big(\gamma (A- z)^{-1}\big)^*$ is bounded, and by replacing $z$ with $\Bar z$ one obtains the 
useful identity
\begin{equation}
	\label{figz}
	\Phi_z^*=\gamma (A- \bar{z})^{-1},\quad z\in\cc\setminus\spec A.
\end{equation}
Now let $\varphi\in C^\infty_0(\rr^2,\cc^2)$ and $h\in L^2(\Sigma,\cc^2)$, then
\begin{align*}
	\big\langle \Phi_z h,(D-\Bar z)\varphi\big\rangle_{L^2(\rr^2,\cc^2)}&=
	\big\langle h,\Phi_z^*(D-\Bar z)\varphi\big\rangle_{L^2(\Sigma,\cc^2)}\\
	&=\big\langle h,\gamma(D-\Bar z)^{-1}(D-\Bar z)\varphi\big\rangle_{L^2(\Sigma,\cc^2)}\\
	&=\big\langle h,\gamma\varphi\big\rangle_{L^2(\Sigma,\cc^2)},
\end{align*}
and it follows that $(D-z)\Phi_z h=0$ in $\cD'(\rr^2\setminus\Sigma)$.
In particular,
\[
\ran \Phi_z\subset 
\big\{f\in H(\sigma,\mathbb R^2\setminus\Sigma):
(D-z)f=0\text{ in }\rr^2\setminus\Sigma\big\}.
\]
In fact, for any $z\in\cc\setminus\spec A$ one has the stronger property \cite[Lemma 4.2]{BHSS}:
\begin{equation}\label{bounPhi}
	\Phi_z:\ L^2(\Sigma,\cc^2)\to H^\half(\sigma,\rr^2\setminus\Sigma) \text{ is bounded.} 
\end{equation}

For all admissible $z$ consider the singular integral operator
\[
\calc_{z}: L^2(\Sigma,\cc^2)\longrightarrow L^2(\Sigma,\cc^2)
\]
given by
\begin{align*}
  \calc_{z}g(x) &= \pv \int_{\Sigma} \phi_{z}(x-y)g(y)\,\mathrm{d}s(y),\quad x\in\Sigma.
\end{align*}
To summarize its properties we introduce the tangent vector field 
\[
\tau=(\tau_1,\tau_2):=(-\nu_2,\nu_1)
\]
on $\Sigma$ and denote
\[
t:= \text{the operator of multiplication by $\tau_1+i \tau_2$ in $L^2(\Sigma)$}.
\]
Then
\begin{equation}\label{Cauchy2}
\begin{aligned}
		C_\Sigma t^*g(x)&= \dfrac{i}{2\pi}\,\pv\int_{\Sigma}\dfrac{g(y)}{(x_1-y_1) +i(x_2-y_2)}\,\mathrm{d}s(y),\\
		t C_\Sigma^*g(x)&= \dfrac{i}{2\pi}\,\pv\int_{\Sigma}\dfrac{g(y)}{(x_1-y_1) -i(x_2-y_2)}\,\mathrm{d}s(y),\quad x\in\Sigma,
	\end{aligned}
\end{equation}
and
\begin{equation}\label{IdenC}
\calc_m= 
	  \begin{pmatrix} 
	       0& C_\Sigma t^*\\ 
	tC_\Sigma^* & 0 \, 
	   \end{pmatrix}.
\end{equation}
Therefore, the boundedness of $C_\Sigma$ implies the boundedness of $\calc_m$. In addition, the expansion \eqref{fifi} shows that
$\calc_z-\calc_m$ is an integral operator with a Hilbert-Schmidt kernel, in particular,
\[
\calc_z-\calc_m:\ L^2(\Sigma,\cc^2)\to L^2(\Sigma,\cc^2) \text{ is compact for any $z\in\cc\setminus \spec A$,}
\]
which also shows the well-definedness  and boundedness of $\calc_z$ for all admissible $z$.

Let $\gamma_\pm:H^\half(\sigma,\Omega_\pm)\to L^2(\Sigma,\cc^2)$ be the trace operators,
and for any $f\in H^\half(\sigma,\rr^2\setminus\Sigma)$ we set
\[
\gamma_\pm f:=\gamma_\pm f_\pm,
\]
then one has the so-called jump formula   
\begin{align}\label{traceof Phi}
	\gamma_\pm \Phi_{z}g=  \left(\mp\frac{i}{2}\sigma\cdot\nu + \calc_z\right)g, \quad g\in L^2(\Sigma,\cc^2).
\end{align}
In \cite[Proposition 3.5]{BHOP} the jump formula was proved under the formal assumption that $\Sigma$ is $C^\infty$ smooth, but the same proof applies to our case as well, as the Plemelj-Sokhotski formula used in the proof
also holds for closed Lipschitz curves. From the jump formula \eqref{traceof Phi} one obtains
\[
g= i(\sigma\cdot\nu)\Big[\gamma_+ \Phi_z g-\gamma_-\Phi_z g \Big],
\quad
g\in L^2(\Sigma,\cc^2),
\]
which shows the injectivity of $\Phi_z$. Further direct consequences of the jump formula are the identities
\begin{equation}
	\label{jump2}
	\begin{aligned}
\gamma_+ \Phi_z g -\gamma_-\Phi_z g&=-i(\sigma\cdot\nu)g,\\
\dfrac{\gamma_+\Phi_z g +\gamma_-\Phi_z g}{2}&=\calc_z g,\quad g\in L^2(\Sigma,\cc^2).
\end{aligned}
\end{equation}

For $z\in(\cc\setminus\spec A)\cup\{m\}$ consider the bounded linear operator
\[
\Theta_z:=I+(\eps\sigma_0+\mu\sigma_3)\calc_{z}:\ L^2(\Sigma,\cc^2)\to L^2(\Sigma,\cc^2),
\]
which is closely related to the operator $B$ from \eqref{bdom} as follows:

\begin{lemma} \label{lem-discr}
For any $z\in\cc\setminus \spec A$	there holds
$\ker(B-z)=\Phi_z\ker\Theta_z$, in particular,
$\dim\ker(B-z)=\dim\ker\Theta_z$.
\end{lemma}

\begin{proof} 
Remark that the last assertion follows from the injectivity of $\Phi_z$.

	Let $z\in\cc\setminus \spec A$ and $g\in\ker \Theta_z$. Denote $f:=\Phi_{z}g$, then $(D-z)f=0$ in $\rr^2\setminus\Sigma$
	due to the above properties of $\Phi_z$. We need to show $f\in\dom B$.
	By \eqref{bounPhi} we have already $f\in H^\half(\sigma,\rr^2\setminus\Sigma)$. By \eqref{jump2}
	we have
\begin{align*}
		(\eps\sigma_0+\mu \sigma_3)	&\dfrac{\gamma_+\Phi_z g +\gamma_-\Phi_z g}{2}
		+i(\sigma\cdot\nu)\big(\gamma_+ \Phi_z g -\gamma_-\Phi_z g\big)\\
		&=(\eps\sigma_0+\mu \sigma_3)\calc_z g+i(\sigma\cdot\nu)\big(-i (\sigma\cdot\nu)\big)g\\
		&=(\eps\sigma_0+\mu \sigma_3)\calc_z g+g=\Theta_z g=0.
\end{align*}
Hence, $f\in\ker(B-z)$. This shows the inclusion $\Phi_z\ker \Theta_z\subset \ker(B-z)$.

Now let $z\in\cc\setminus \spec A$ and $f\in \ker (B-z)$. Due to \eqref{distjump} we have
	\begin{equation}
		\label{temp5}
	(D-z)f=(B-z)f+i(\sigma\cdot\nu)(f_+ - f_-)\delta_{\Sigma}.
	\end{equation}
Let $\cF:\cS'(\rr^2)\to\cS'(\rr^2)$	be the Fourier transform. For any $\psi\in\cS'(\rr^2)$ we have 
\[
\cF (D-z)\psi=(\sigma\cdot\xi+m\sigma_3-z\sigma_0)\cF \psi.
\]
The matrix $\sigma\cdot\xi+m\sigma_3-z\sigma_0$ is invertible for any $\xi\in\rr^2$ and has polynomial entries, which shows that $D-z:\cS'(\rr^2)\to\cS'(\rr^2)$ is injective. As the function $\phi_z\in \cS'(\rr^2)$ is a fundamental solution of $D-z$, from \eqref{temp5} one obtains
\[
f=\phi_z\ast \big[i(\sigma\cdot\nu)(f_+ - f_-)\delta_{\Sigma}\big].
\]
Due to $f\in\dom B$ we have $f_\pm\in H^\half(\sigma,\Omega_\pm)$, and, hence 

\[
g:=i(\sigma\cdot\nu)(\gamma_+ f - \gamma_- f)\in L^2(\Sigma,\cc^2).
\] Then
\[
f=\phi_z\ast (g\delta_\Sigma)=\int_\Sigma\phi_z(\cdot-y)g(y)\,\mathrm{d}s(y)\equiv \Phi_z g.
\]
With the help of \eqref{jump2} we obtain
\begin{align*}
	0&=(\eps\sigma_0+\mu\sigma_3)\dfrac{\gamma_+ f + \gamma_- f}{2}
	+i(\sigma\cdot\nu)(\gamma_+ f - \gamma_- f)\\
	&=(\eps\sigma_0+\mu\sigma_3)\calc_{z} g+ g=\Theta_z g,
\end{align*}
which implies $g\in \ker\Theta_z$. Hence, $\ker(B-z)\subset\Phi_z\ker\Theta_z$.	
\end{proof}

For the sake of completeness, we include the proof of the following important statement (which is based on similar ideas):
\begin{lemma}\label{lem22}
The operator $C_\Sigma^2-\frac{1}{4}$ is compact in $L^2(\Sigma)$.	
\end{lemma}

\begin{proof}
Let $h\in L^2(\Sigma,\cc^2)$ and $z\in\cc\setminus\spec A$. Consider $f:=\Phi_z h$,
then $(D-z)f=0$ in $\Omega_\pm$. Consider further the function
\[
\widetilde f: \ \rr^2\ni x\mapsto \begin{cases}
	f(x), & x\in \Omega_+,\\
	0, & \text{otherwise}.
	\end{cases}
\]
One has $\gamma_+\widetilde f=\gamma_+ f$ and $\gamma_- \widetilde f=0$, with $(D-z)\widetilde f=0$ in $\Omega_\pm$,
and \eqref{distjump} gives
\[
(D-z)\widetilde f=i(\sigma\cdot \nu)(\gamma_+\widetilde f-\gamma_-\widetilde f)\delta_\Sigma\equiv i(\sigma\cdot \nu)\gamma_+ f\,\delta_\Sigma \text{ in } \cD'(\rr^2),
\]
which implies $\widetilde f=\phi_z\ast \big[i(\sigma\cdot \nu)\gamma_+ f\,\delta_\Sigma\big]\equiv
\Phi_z i(\sigma\cdot \nu)\gamma_+ f$. In particular,
\begin{equation}
	\label{fff3}
\Phi_z i(\sigma\cdot \nu)\gamma_+ f =f=\Phi_z h \text{ in } \Omega_+.
\end{equation}
Remark that by the construction of $f$ we have
\[
\gamma_+f=\Big(-\dfrac{i (\sigma\cdot\nu )}{2}+\calc_z\Big) h.
\]
Use this last equality in \eqref{fff3} and then apply $\gamma_+$ on the both sides, then one arrives at
\[
\Big(-\frac{i(\sigma\cdot\nu)}{2}+\calc_z\Big)i(\sigma\cdot \nu)\Big(-\dfrac{i (\sigma\cdot\nu )}{2}+\calc_z\Big) h=\Big(-\dfrac{i (\sigma\cdot\nu )}{2}+\calc_z\Big) h,
\]
which after a simple computation takes the form
\[
\calc_zi(\sigma\cdot\nu)\calc_z h=-\dfrac{i(\sigma\cdot\nu)}{4}\, h,
\]
and results in the identity
\begin{equation}
	\label{cc2}
	\Big(\calc_z (\sigma\cdot\nu)\Big)^2=-\frac{1}{4}\, I.
\end{equation}
The identities are well-known for the three-dimensional case \cite[Lemma 3.3]{AMV1}, but we gave a complete argument
to stay self-contained. Further remark that
\[
\sigma\cdot\nu=\begin{pmatrix}
	0 & n^*\\
	n & 0
\end{pmatrix},
\]
where $n$ is the operator of multiplication by $\nu_1+i\nu_2$. Using \eqref{IdenC} we write
\[
\calc_z=\begin{pmatrix} 
		0& C_\Sigma t^*\\ 
		tC_\Sigma^* & 0 \, 
	\end{pmatrix}+
	M_0
\]
with a compact operator $M_0$. We have $t^*n=-i I$, so the substitution into \eqref{cc2} gives,
with some compact operators $M_j$,
\[
-\dfrac{1}{4}\, I=
\left[\begin{pmatrix}
	-iC_\Sigma & 0\\
	0 &tC_\Sigma^* n^*
\end{pmatrix}+M_1\right]^2=\begin{pmatrix}
-C_\Sigma^2 & 0\\
0 & (tC_\Sigma^* n^*)^2
\end{pmatrix}
+M_2,
\]
and the upper left block gives the sought result.
\end{proof}

\section{Case $|\eps|=|\mu|$}

We first consider the self-adjointness of $B$ for $|\eps|=|\mu|$.

\begin{theorem}\label{thm31}
	The operator $B$ in \eqref{bdom} is self-adjoint for $|\eps|=|\mu|$.
\end{theorem}

\begin{proof}
In the case $\eps=\mu=0$ we obviously have  $B=A$. From now on let
\[
\mu=\pm \eps \text{ with }\eps\ne 0.
\]
Consider the following maps
\begin{align*}
P_+:&\ L^2(\Sigma)\ni f\mapsto \begin{pmatrix}
	f\\
	0
\end{pmatrix}\in L^2(\Sigma,\cc^2),\\
P_-:&\ L^2(\Sigma)\ni f\mapsto \begin{pmatrix}
	0\\
	f
\end{pmatrix}\in L^2(\Sigma,\cc^2),
\end{align*}
and their adjoints
\begin{align*}
P_+^*:&\ L^2(\Sigma,\cc^2)\ni \begin{pmatrix} f_1\\f_2\end{pmatrix}\mapsto f_1 \in L^2(\Sigma),\\
P_-^*:&\ L^2(\Sigma,\cc^2)\ni \begin{pmatrix} f_1\\f_2\end{pmatrix}\mapsto f_2 \in L^2(\Sigma).
\end{align*}
We set
\[
P:=P_\pm \text{ for }\eps=\pm\mu.
\]

As the operator $B$ is symmetric, it is sufficient to show that $\ran(B-z)=L^2(\rr^2,\cc^2)$ for any $z\in\cc\setminus\rr$.
For that, we will explicitly construct the inverse $(B-z)^{-1}$.

Let $z\in \cc\setminus\rr$. As $B$ is symmetric, $\ker(B-z)=\{0\}$, and Lemma \ref{lem-discr} implies $\ker \Theta_z=\{0\}$.
Remark that in the present case, we have
\begin{gather*}
\Theta_z=I+2\eps P P^*\calc_z,
\quad
\Theta_zP =P +2\eps P P^*\calc_zP\equiv 2\eps P \lambda_z\\
\text{for }
\lambda_z:=\dfrac{1}{2\eps}\,I+P^*\calc_zP:\ L^2(\Sigma)\to L^2(\Sigma)\equiv \dfrac{1}{2\eps}\,I+(z\pm m)S_z
\end{gather*}
with the operator $S_z:L^2(\Sigma)\to L^2(\Sigma)$ given by
\[
(S_z g)(x):=\dfrac{1}{2\pi}\int_\Sigma K_0\big(\sqrt{m^2-z^2}|x-y|\big) g(y)\,\mathrm{d}s(y),\quad x\in\Sigma,\quad g\in L^2(\Sigma).
\]
The integral kernel of $S_z$ has a logarithmic singularity on the diagonal, therefore, $S_z$ is Hilbert-Schmidt (in particular, compact). It follows that $\lambda_z$ is a Fredholm operator of index zero. From the injectivity of $\Theta_z$
and $P$ one obtains the injectivity of $\lambda_z$, and it follows that $\lambda_z:L^2(\Sigma)\to L^2(\Sigma)$ is bijective.

Now we are going to show that the operator
\[
R(z):=(A-z)^{-1}-\Phi_z P \lambda_z^{-1} P^*\Phi^*_{\Bar z},
\]
is the inverse of $B-z$. Let $v\in L^2(\rr^2,\cc^2)$. Due to the mapping properties of $\Phi_z$ in \eqref{bounPhi} and those of $(A-z)^{-1}$
one has
\[
f:=R(z)v\in H^\half(\rr^2\setminus\Sigma,\cc^2).
\]
Using the jump formulas \eqref{jump2} we obtain
\begin{align*}
	\dfrac{\gamma_+ f + \gamma_- f}{2}&=\gamma(A-z)^{-1}v-\calc_z P \lambda_z^{-1} P^*\Phi^*_{\Bar z} v
	\equiv \Phi^*_{\Bar z} v-\calc_z P \lambda_z^{-1} P^*\Phi^*_{\Bar z} v,\\
	\gamma_+ f - \gamma_- f&=i(\sigma\cdot\nu)P \lambda_z^{-1} P^*\Phi^*_{\Bar z} v.
\end{align*}
We have then
\begin{align*}
	(\eps\sigma_0+\mu\sigma_3)&\dfrac{\gamma_+ f+\gamma_- f}{2}+i(\sigma\cdot\nu)(\gamma_+f-\gamma_- f)\\
	&\equiv
	2\eps P P^*\dfrac{\gamma_+ f+\gamma_- f}{2}+i(\sigma\cdot\nu)(\gamma_+f-\gamma_- f)\\
	&=
	2\eps P P^*\big( \Phi^*_{\Bar z} v-\calc_z P \lambda_z^{-1} P^*\Phi^*_{\Bar z} v\big)
	+i(\sigma\cdot\nu)i(\sigma\cdot\nu)P \lambda_z^{-1} P^*\Phi^*_{\Bar z} v\\
	&=2\eps P P^*\big( \Phi^*_{\Bar z} v-\calc_z P \lambda_z^{-1} P^*\Phi^*_{\Bar z} v\big)-P \lambda_z^{-1} P^*\Phi^*_{\Bar z} v\\
	&=P\big( 2\eps I -2\eps P^*\calc_z P \lambda_z^{-1}-\lambda_z^{-1}\big) P^*\Phi^*_{\Bar z} v,
\end{align*}
while
\begin{align*}
2\eps -2\eps P^*\calc_z P \lambda_z^{-1}-\lambda_z^{-1}&=2\eps I-2\eps \Big( P^*\calc_z P +\frac{1}{2\eps}\,I\Big)\lambda_z^{-1}\\
&=2\eps I-2\eps \lambda_z \lambda_z^{-1}=0.
\end{align*}
This shows that $f$ satisfies the transmission condition \eqref{TC} and, therefore, $f\in\dom B$.

Further, in $\cD'(\rr^2\setminus \Sigma,\cc^2)$ we have $(D-z)\Phi_z P \lambda_z^{-1} P^*\Phi^*_{\Bar z}v=0$, therefore,
\[
(B-z)f=(D-z)f=(D-z)(A-z)^{-1} v=(A-z)(A-z)^{-1}v=v,
\]
which shows $R(z)=(B-z)^{-1}$.
\end{proof}

\section{Case $|\eps|\ne|\mu|$}

For $|\eps|\ne|\mu|$ the matrix $\eps\sigma_0+\mu\sigma_3$ is invertible, with
\[
(\eps\sigma_0+\mu\sigma_3)^{-1}=\dfrac{1}{\eps^{2} - \mu^{2}}(\eps\sigma_0 - \mu\sigma_3),
\]
and it will be more convenient to consider the auxiliary bounded linear operators
\begin{align*}
    \Lambda_z:=\dfrac{1}{\eps^{2} - \mu^{2}}(\eps\sigma_0 - \mu\sigma_3) + \calc_z\equiv
    (\eps \sigma_0+\mu\sigma_3)^{-1} \Theta_z
\end{align*}
for $z\in(\cc\setminus\spec A)\cup\{m\}$. The symmetry property $\phi_{z}(y-x)^{\ast}=\phi_{\overline{z}}(x-y)$ entails that
both $\calc_z$ and $\Lambda_z$ are self-adjoint for real admissible $z$.

The following assertion can be viewed as a simplified version of the results of \cite{BHSS}, and this is the entry point
for the subsequent analysis:

\begin{theorem}\label{prop21}
Let $|\eps|\ne|\mu|$ such that the operator $\Lambda_a$ is Fredholm for some $a\in(\cc\setminus\spec A)\cup\{m\}$, then the operator $B$ in \eqref{bdom} is self-adjoint.
\end{theorem}

\begin{proof}
%
Let $\Lambda_a$ be Fredholm. As noted above, for any $z\in\cc\setminus\spec A$ the difference $\Lambda_z-\Lambda_a\equiv \calc_z-\calc_a$ is a compact operator, and it follows that $\Lambda_z$ is also Fredholm and has the same index as $\Lambda_a$.

Now let $z\in\big(-|m|,|m|\big)\cup\{m\}$, then $\Lambda_z$ is self-adjoint. From the Fredholmness and the self-adjointness, it follows that the index of $\Lambda_z$ is zero. We have just seen above that the index is independent of $z$, so $\Lambda_z$ is Fredholm of index zero for all $z\in\cc\setminus\spec A$.

As $B$ is symmetric, in order to show its self-adjointness it is sufficient to show that $\ran (B-z)=L^2(\rr^2,\cc^2)$ for all $z\in\cc\setminus\rr$. We will do it by constructing explicitly the inverse $(B-z)^{-1}$ defined on $L^2(\rr^2,\cc^2)$.

Let $z\in\cc\setminus\rr$. As $B$ is symmetric, there holds $\ker(B-z)=\{0\}$. By Lemma \ref{lem-discr} one obtains $\ker \Lambda_z=\{0\}$. As $\Lambda_z$ is Fredholm of index zero, one has $\ran \Lambda_z=L^2(\Sigma,\cc^2)$,
so $\Lambda_z:L^2(\Sigma,\cc^2)\to L^2(\Sigma,\cc^2)$ is bijective with a bounded inverse.
Consider the bounded linear operator
\begin{align*}\label{Kresolvent1}
	R(z)= (A-z)^{-1} - \Phi_{z}\Lambda_{z}^{-1} \Phi^{\ast}_{\overline{z}}:\ L^2(\rr^2,\cc^2)\to L^2(\rr^2,\cc^2).
\end{align*}
We are going to show that $R(z)=(B-z)^{-1}$.

Let $v\in L^2(\rr^2,\cc^2)$. Due to \eqref{bounPhi} one has
\[
f:=R(z)v\in H^\half(\rr^2\setminus\Sigma,\cc^2).
\]
Using \eqref{jump2} we obtain
\begin{align*}
\dfrac{\gamma_+f + \gamma_- f}{2}&=\gamma(A-z)^{-1} v -\calc_z \Lambda_{z}^{-1} \Phi^*_{\overline{z}}v
=\Phi^*_{\overline{z}}v-\calc_z \Lambda_{z}^{-1} \Phi^*_{\overline{z}}v,\\
\gamma_+ f - \gamma_- f&=i\,(\sigma\cdot\nu) (\Lambda_{z} )^{-1} \Phi^*_{\overline{z}}v.
\end{align*}
Then
\begin{align*}
	(\eps \sigma_0 + &\mu\sigma_{3})\frac{\gamma_+f + \gamma_-f}{2}+  i(\sigma\cdot\nu )(\gamma_+f - \gamma_-f)\\
	&=\Big[(\eps \sigma_0 + \mu\sigma_{3})(I-\calc_z \Lambda_z^{-1})-\Lambda_z^{-1}\Big] \Phi^*_{\overline{z}}v,
\end{align*}
while
\begin{align*}
	(\eps \sigma_0 + \mu\sigma_{3})(I-\calc_z \Lambda_z^{-1})-\Lambda_z^{-1}&=
	\Big[ 
	(\eps \sigma_0 + \mu\sigma_{3})(\Lambda_z-\calc_z)-I	
	\Big]\Lambda_z^{-1}\\
	&=	
	\Big[(\eps \sigma_0 + \mu\sigma_{3})\dfrac{1}{\eps^{2} - \mu^{2}}(\eps\sigma_0 - \mu\sigma_3)-I\Big]\Lambda_z^{-1}\\
	&	=(I-I)\Lambda_z^{-1}=0.
\end{align*}
This shows that $f$ satisfies the transmission condition \eqref{TC}, i.e. $f\in\dom B$.
In addition, in ${\mathcal D}'(\rr^2\setminus\Sigma)$ we have $(D-z)\Phi_{z}\Lambda_{z}^{-1} \Phi^{\ast}_{\overline{z}}=0$, therefore,
\begin{align*}
(B-z)f&=(D-z)f=(D-z)R(z)v\\
&=(D-z)(A-z)^{-1}v=(A-z)(A-z)^{-1}v=v,
\end{align*}
which shows the required identity $R(z)=(B-z)^{-1}$.
\end{proof}

The following theorem gives a precise range of $(\eps,\mu)$ for which $B$ is self-adjoint without additional assumptions on $\Sigma$.

\begin{theorem}\label{thm43} Assume that  $|\eps| <|\mu|$, then $B$ is self-adjoint.
\end{theorem}
\begin{proof}By Theorem~\ref{prop21} it is sufficient to show that $(\eps^2-\mu^2)\Lambda_m$ is Fredholm. Using \eqref{IdenC} we represent
\begin{align*}
	(\eps^2-\mu^2)\Lambda_m&=(\eps\sigma_0-\mu\sigma_3)+(\eps^2-\mu^2)\calc_m\\
	&=(\eps\sigma_0-\mu\sigma_3)+(\eps^2-\mu^2)\begin{pmatrix} 
		0& C_\Sigma t^*\\ 
		tC_\Sigma^* & 0 
	\end{pmatrix}=\eps\sigma_0  + \Gamma,\\
	\text{with }
	\Gamma&:=\begin{pmatrix}
		-\mu I& (\eps^2-\mu^2) C_\Sigma t^*\\ 
		(\eps^2-\mu^2)tC_\Sigma^* & \mu I
	\end{pmatrix}.
\end{align*}	
Remark that $\Gamma$ is self-adjoint and
\[
\Gamma^2=\mu^2 I+(\eps^2-\mu^2)^2\begin{pmatrix}
	C_\Sigma C_\Sigma^* & 0\\
	0 & t C_\Sigma^* C_\Sigma t^*
\end{pmatrix}.
\]
The last term is a non-negative operator, which shows
\[
\spec(\Gamma^2)\subset [\mu^2,\infty),
\quad
\spec \Gamma \cap \big(-|\mu|,|\mu|\big)=\emptyset.
\]
Therefore, if $|\eps|<|\mu|$, then the operator
\[
(\eps^2-\mu^2)\Lambda_m\equiv\eps I+\Gamma:\ L^2(\Sigma,\cc^2)\to L^2(\Sigma,\cc^2)
\]
is an isomorphism and, in particular, Fredholm.	
\end{proof}

By summarizing Theorems \ref{thm31} and \ref{thm43} we arrive at
\begin{corollary}\label{bleq}
	The operator $B$ is self-adjoint for any $(\eps,\mu)$ with $|\eps|\le|\mu|$.
\end{corollary}

Remark that the preceding discussion is valid without any additional assumptions on $\Sigma$ (i.e. only assumes that $\Sigma$ is Lipschitz). Under stronger geometric assumptions one can indeed enlarge the range of parameters for which the self-adjointness is guaranteed. The following result follows implicitly from the machinery of \cite{BHSS}, but we prefer to give an explicit formulation with a direct argument.

\begin{theorem}\label{thmc1}
	If $\Sigma$ is $C^1$-smooth and $\eps^2-\mu^2\ne 4$, then $B$ is self-adjoint.
\end{theorem}

\begin{proof}
	The case $|\eps|=|\mu|$ is already covered by Theorem \ref{thm31}, so from now on assume $|\eps|\ne|\mu|$.	
	By Theorem~\ref{prop21} it is sufficient to show that $\Lambda_m$ is Fredholm.
	Due to the self-adjointness of $\Lambda_m$ this is equivalent to
	\begin{equation}
		\label{lmm}
		0\notin\spece(\eps^2-\mu^2)\Lambda_m.
	\end{equation}
	Using \eqref{IdenC} we represent
	\begin{align*}
		(\eps^2-\mu^2)\Lambda_m&=(\eps\sigma_0-\mu\sigma_3)+(\eps^2-\mu^2)\calc_m\\
		&=(\eps\sigma_0-\mu\sigma_3)+(\eps^2-\mu^2)\begin{pmatrix} 
			0& C_\Sigma t^*\\ 
			tC_\Sigma^* & 0 
		\end{pmatrix}=\eps\sigma_0  + \Gamma,\\
		\text{with }
		\Gamma&:=\begin{pmatrix}
			-\mu I& (\eps^2-\mu^2) C_\Sigma t^*\\ 
			(\eps^2-\mu^2)tC_\Sigma^* & \mu I
		\end{pmatrix}.
	\end{align*}
	By \cite[Theorem 3.2]{Lan} the operator $C_\Sigma-C_\Sigma^*$ is compact, therefore,
	\[
	\Gamma=\begin{pmatrix}
		-\mu I& (\eps^2-\mu^2) C_\Sigma t^*\\ 
		(\eps^2-\mu^2)tC_\Sigma & \mu I
	\end{pmatrix}+M_0
	\]
	with some compact operator $M_0$. Using Lemma \ref{lem22} we obtain,
	with some compact operators $M_1$ and $M_2$,
	\begin{align*}
		\Gamma^2&=\mu^2 I+(\eps^2-\mu^2)^2\begin{pmatrix}
			C_\Sigma^2 & 0\\
			0 & t C_\Sigma^2 t^*
		\end{pmatrix}+M_1\\
		&\equiv
		\mu^2 I+\dfrac{(\eps^2-\mu^2)^2}{4}\begin{pmatrix}
			I & 0\\
			0 & I
		\end{pmatrix}+M_2.
	\end{align*}
	It follows that
	\[
	\spece (\Gamma^2)=\Big\{\mu^2+\dfrac{(\eps^2-\mu^2)^2}{4}\Big\},
	\]
	and the self-adjointness of $\Gamma$ implies
	\[
	\spece\Gamma\subset\bigg\{\,
	-\sqrt{\mu^2+\dfrac{(\eps^2-\mu^2)^2}{4}},
	\sqrt{\mu^2+\dfrac{(\eps^2-\mu^2)^2}{4}}\,
	\bigg\}.
	\]
	Due to the above identity $(\eps^2-\mu^2)\Lambda_m=\eps I+\Gamma$ the condition \eqref{lmm} is satisfied for
	\[
	|\eps|\ne\sqrt{\mu^2+\dfrac{(\eps^2-\mu^2)^2}{4}},\,
	\text{ i.e. }\,
	\eps^2-\mu^2\ne \dfrac{(\eps^2-\mu^2)^2}{4}
	\]
	which reduces to $\eps^2-\mu^2\ne4$.
\end{proof}

\section{Fredholmness for curvilinear polygons\footnote[1]{In the published journal version, Section~5 contains a serious mistake. The arguments were based on the computation of the essential spectrum of a singular integral operator by applying the approach of \cite{VYS}, but the assumptions for this application were not checked carefully: in fact, the integral kernel considered is not covered by~\cite{VYS}, and this invalidates all results stated in Section~5 in the published version (the results in Sections 2--4 remain unaffected). Here we present a correct version.}}

Assume that $\Sigma$ is a closed piecewise $C^1$-smooth Lipschitz curve, with finitely many corner points $a_1,\dots,a_n$. For each corner $a_j$, let
\[
\theta_j\in(0,2\pi)\setminus\{ \pi \}
\]
be the non-oriented interior angle of $\Sigma$ at the point $a_j$ measured inside $\Omega_+$.
Denote by $\omega$ the smallest angle of $\Sigma$,
\[
\omega:=\min_{j\in\{1,\dots,n\}} \min\{\theta_j,2\pi-\theta_j\}\in(0,\pi).
\]
Recall that in \cite[Thm.~2]{bolt}
it was shown that
\begin{gather}
	\label{spece}
	\spece i(C_\Sigma-C_\Sigma^*) =\big[-\kappa(\omega),\kappa(\omega)\big],\\
	\label{k1}
	\kappa(\omega):=\sup_{x\ge 0}\dfrac{\sinh\big((\pi-\omega)x\big)}{\cosh (\pi x)}.
\end{gather}
We further recall that by \cite[Thm.~4.3]{Lan} we have $(\frac{1}{2}I-C_\Sigma)^2=\frac{1}{2}I-C_\Sigma$, which implies
\begin{equation}
	\label{ccc}
	C_\Sigma^2=(C_\Sigma^*)^2=\tfrac{1}{4}I,
	\quad
	(C_\Sigma C_\Sigma^*) (C^*_\Sigma C_\Sigma)=(C_\Sigma^* C_\Sigma) (C_\Sigma C_\Sigma^*)=\tfrac{1}{16}I.
\end{equation}

\begin{lemma}\label{lem-1}
	For $|\eps|\ne|\mu|$ the operator $\Lambda_m$ is Fredholm if and only if
	\begin{equation}
		\label{ccstar}
		\tfrac{1}{\eps^2-\mu^2}\notin \spece C_\Sigma^* C_\Sigma.
	\end{equation}
\end{lemma}

\begin{proof}
	In view of the block representation
	\[
	\Lambda_m=\begin{pmatrix}
		\tfrac{1}{\eps+\mu} I & C_\Sigma t^*\\
		tC_\Sigma^* & \tfrac{1}{\eps-\mu} I
	\end{pmatrix}
	\]
	and the invertibility of the upper left block, the Fredholmness
	of $\Lambda_m$ is equivalent to the Fredholmness of the respective Schur complement
	\[
	S:=\tfrac{1}{\eps-\mu} I-tC_\Sigma^*\Big(\tfrac{1}{\eps+\mu} I\Big)^{-1}C_\Sigma t^*\equiv(\eps+\mu)\big(\tfrac{1}{\eps^2-\mu^2}I-t C_\Sigma^* C_\Sigma t^*\big),
	\]	
	see e.g.~\cite[Thm.~2.4.6]{Tretter}. As $S$ is self-adjoint, its Fredholmness is equivalent
	to $0\notin\spece S$, which takes the form \eqref{ccstar} in view of the unitarity of $t$.
\end{proof}

\begin{lemma}\label{lem-2}
	With  $\kappa$ defined in \eqref{k1} one has
	\begin{equation*}
		\spece C_\Sigma^* C_\Sigma=\Big[\dfrac{1}{16 m(\omega)},m(\omega)\Big],\quad
		m(\omega):=\dfrac{\big(\sqrt{1+\kappa(\omega)^{2\mathstrut}}+\kappa(\omega)\big)^2}{4}.
	\end{equation*}
\end{lemma}

\begin{proof}
	Denote 
	\[
	T:=C_\Sigma^*C_\Sigma>0,\quad Q:=i(C_\Sigma-C_\Sigma^*),\quad \kappa:=\kappa(\omega),
	\]
	then the identities \eqref{ccc} yield
	\begin{align*}
		Q^2&=-(C_\Sigma)^2 +C_\Sigma C_\Sigma^*+C_\Sigma^* C_\Sigma -(C_\Sigma^*)^2
		=-\dfrac{1}{2}I+\dfrac{1}{16} T^{-1}+T,
	\end{align*}
	which can be rewritten as
	\[
	g(T)=\dfrac{1}{2}I+Q^2,\qquad g:\ (0,\infty)\ni x\mapsto \dfrac{1}{16x}+x\in\rr.
	\]
	Due to \eqref{spece} we arrive at
	\[
	\spece g(T)=\big[\tfrac{1}{2}, \tfrac{1}{2}+\kappa^2\big]=:J.
	\]
	
	We have
	\begin{align*}
		X:=g^{-1}(J)&=\big\{x>0:\ \tfrac{1}{16 x}+x\in J\big\}\\
		&=\big\{x>0:\ \exists a\in[0,\kappa] \text{ with } \tfrac{1}{16 x}+x=\tfrac{1}{2}+a^2\big\}\\
		&=\big\{x>0:\ \exists a\in[0,\kappa] \text{ with } x^2-\big(\tfrac{1}{2}+a^2\big)x+\tfrac{1}{16}=0\big\}.
	\end{align*}
	For each $a\in[0,\kappa]$ the equation
	\[
	x^2-\big(\tfrac{1}{2}+a^2\big)x+\tfrac{1}{16}=0
	\]
	has the solutions $x_\pm(a)>0$ given by
	\begin{align*}
		x_\pm(a)&:=\dfrac{\dfrac{1}{2}+a^2\pm \sqrt{\Big(\dfrac{1}{2}+a^2\Big)^2-4\cdot\dfrac{1}{16}}}{2}
		=\dfrac{1+2a^2\pm 2\sqrt{a^2(1+a^2)}}{4}\\
		&=\dfrac{1+2a^2\pm 2a\sqrt{1+a^2}}{4}
		=\dfrac{(\sqrt{1+a^2}\pm a)^2}{4}.
	\end{align*}
	It follows that $X=X_-\,\cup\,X_+$ with $X_\pm:=x_\pm\big([0,\kappa]\big)$, i.e.
	\begin{align*}
		X_-&=\Big[\frac{(\sqrt{1+\kappa^2}-\kappa)^2}{4},\frac{1}{4}\Big]\equiv \Big[\frac{1}{16m},\frac{1}{4}\Big],&
		m&:=\dfrac{(\sqrt{1+\kappa^2}+\kappa)^2}{4},\\
		X_+&=\Big[\frac{1}{4},\frac{(\sqrt{1+\kappa^2}+\kappa)^2}{4}\Big]=\Big[\frac{1}{4},m\Big],
		&
		X&=\Big[\dfrac{1}{16m},m\Big].
	\end{align*}
	
	The equality $g(\spece T)=\spece g(T)=J$ yields
	$\spece T\subset X$. To show the reverse inclusion we note that by construction the maps
	\[
	x_\pm: [0,\kappa]\to X_\pm,
	\quad
	g:\ X_\pm\to J,
	\]
	are bijective, with
	\[
	g\big(x_-(a)\big)=g\big(x_+(a)\big),\quad x_-(a)x_+(a)=\tfrac{1}{16} \text{ for all }a\in[0,\kappa],
	\]
	and it follows that for each $a\in[0,\kappa]$ the set $\spece T$ contains at least one of the points $x_\pm(a)\in X_\pm$.
	At the same time,
	\[
	C_\Sigma T C_\Sigma^{-1}=C_\Sigma C_\Sigma^* C_\Sigma C_\Sigma^{-1}=C_\Sigma C_\Sigma^*=\tfrac{1}{16}T^{-1},
	\]
	and in view of $\spece T=\spece C_\Sigma T C_\Sigma^{-1}$ we conclude that
	\[
	\spece T=\spece \tfrac{1}{16}T^{-1}=\big\{\tfrac{1}{16 x}:\ x\in \spece T\big\}.
	\]
	It follows that if $x_\pm(a)\in\spece T$, then also $x_\mp(a)\equiv \frac{1}{16 x_\pm(a)}\in\spece T$.
	This yields $\big\{x_-(a),x_+(a)\big\}\subset\spece T$ for all $a\in[0,\kappa]$, hence, $X\subset \spece T$.	
\end{proof}

\begin{theorem}\label{thm-3}
	If $\Sigma$ is a curvilinear polygon with $C^1$-smooth edges and the smallest angle $\omega\in(0,\pi)$, 
	and
	\[
		\eps^2-\mu^2\notin\Big[\dfrac{1}{m(\omega)},16m(\omega)\Big]
	\]
	with $m(\omega)$ from Lemma~\ref{lem-2},
	then the operator $B$ is self-adjoint.	
\end{theorem}

\begin{proof}
	The case $|\eps|=|\mu|$ is already covered in Theorem \ref{thm31}. For $|\eps|\ne|\mu|$,
	under the stated assumptions one has $\frac{1}{\eps^2-\mu^2}\notin \spece C_\Sigma^* C_\Sigma$ by Lemma~\ref{lem-2}, hence, Lemma~\ref{lem-1} yields the Fredholmness of $\Lambda_m$, and the self-adjointness of $B$
	follows by Theorem \ref{prop21}.
\end{proof}

In order to have more precise results we investigate the dependence of $\kappa(\omega)$ and $m(\omega)$ on $\omega$.

\begin{lemma}\label{lem-4}
	The function $\kappa$ in \eqref{k1} 
	is strictly decreasing with
	\begin{equation}
		\label{kkk}
		\lim_{\omega\to 0^+}\kappa(\omega)=1,\quad
		\lim_{\omega\to \pi^-}\kappa(\omega)=0,\quad
		\kappa\big(\tfrac{\pi}{2}\big)=\tfrac{1}{2\sqrt{2}}.
	\end{equation}
\end{lemma}

\begin{proof}
	Recall that
	\begin{gather*}
		\kappa(\omega)=\sup_{x\ge 0}f_\omega(x),\qquad
		\quad f_\omega:\ x\mapsto \dfrac{\sinh\big((\pi-\omega)x\big)}{\cosh (\pi x)}.
	\end{gather*}
	As $f_\omega(0)=0=\lim\limits_{x\to\infty}f_\omega(x)$, one has $\kappa(\omega)=f_\omega(x_\omega)$
	for some $x_\omega\in(0,\infty)$. For any $0<\omega<\omega'<\pi$, using the strict monotonicity of $\sinh$ we obtain
	\begin{align*}
		\kappa(\omega')&=f_{\omega'}(x_{\omega'})=\dfrac{\sinh\big((\pi-\omega')x_{\omega'}\big)}{\cosh (\pi x_{\omega'})}\\
		&<
		\dfrac{\sinh\big((\pi-\omega)x_{\omega'}\big)}{\cosh (\pi x_{\omega'})}=f_\omega(x_{\omega'})\le\kappa(\omega),
	\end{align*}
	which shows that $\kappa$ is strictly decreasing.
	
	For any $x\ge 0$ we obviously have $f_\omega(x)\ge 0$ and 
\begin{align}
	f_\omega(x)&= \tanh(\pi x)\cosh(\omega x)-\sinh(\omega x)\label{fsh}\\
	&\le \cosh(\omega x)-\sinh(\omega x)=e^{-\omega x}\le 1,\nonumber
\end{align}
hence, $0\le\kappa(\omega)\le 1$.

To study the limit for $\omega\to 0^+$ pick any $\delta>0$. As $\lim\limits_{x\to\infty}\tanh (\pi x)=1$, we may first choose a sufficiently large $r>0$ to have $\tanh(\pi r)>1-\delta$, then choose a small $\omega_0>0$ to have $\sinh(\omega r)<\delta$ for all $\omega\in(0,\omega_0)$. Using \eqref{fsh}, for all $\omega\in(0,\omega_0)$
we arrive
at
\[
\kappa(\omega)\ge f_\omega(r)\ge \tanh(\pi r)-\sinh(\omega r)\ge (1-\delta)-\delta=1-2\delta.
\]
As $\delta$ was arbitrary and $\kappa(\omega)\le 1$, this yields the first limit in \eqref{kkk}.

	For the limit with $\omega\to \pi^-$ we use the convexity of $\sinh$ as follows: for any $x\ge 0$ and $\omega\in(0,\pi)$ we have
	\begin{align*}
		\sinh\big((\pi-\omega)x\big)&=\sinh\big(\tfrac{\omega}{\pi}\cdot 0+\tfrac{\pi-\omega}{\pi} \cdot\pi x\big)\\
		&\le \tfrac{\omega}{\pi}\sinh (0) + \tfrac{\pi-\omega}{\pi}\sinh\big(\pi x\big)=\tfrac{\pi-\omega}{\pi}\sinh\big(\pi x\big),\\
		f_\omega(x)&\le \tfrac{\pi-\omega}{\pi}\tanh\big(\pi x\big)\le \tfrac{\pi-\omega}{\pi},
	\end{align*}
	and by combining with $f_\omega\ge 0$ we arrive at the second limit in \eqref{kkk}.
	
	Finally we note that
	\[
	f_\frac{\pi}{2}(x)=\dfrac{\sinh\frac{\pi x}{2}}{\cosh (\pi x)}=\dfrac{\sinh\frac{\pi x}{2}}{1+2\sinh^2\frac{\pi x}{2}}.
	\]
	As $[0,\infty)\ni x\mapsto \sinh\frac{\pi x}{2}\in[0,\infty)$ is a bijection, we arrive at
	\[
	\kappa\big(\tfrac{\pi}{2}\big)=\max_{x\ge 0}\dfrac{\sinh\frac{\pi x}{2}}{1+2\sinh^2\frac{\pi x}{2}}=\max_{t\ge 0}\dfrac{t}{1+2t^2}.
	\]
	We have $1+2t^2\ge 2\sqrt{1\cdot 2t^2}=2\sqrt{2}t$, which is attained for $1=2t^2$, and the last identity in \eqref{kkk} follows.
\end{proof}

A direct substitution yields then:
\begin{corollary}\label{corol-5}
	The function $m$ in Lemma~\ref{lem-2} is strictly decreasing with
	\[
	\lim_{\omega\to 0^+}m(\omega)=\tfrac{3+2\sqrt{2}}{4},\quad \lim_{\omega\to \pi^-}m(\omega)=\tfrac{1}{4},\qquad m\big(\tfrac{\pi}{2}\big)=\tfrac{1}{2}.
	\]
\end{corollary}	

By combining Theorem \ref{thm-3} with Corollary~\ref{corol-5} we obtain:	
\begin{corollary}\label{corol-6}
	Let $\Sigma$ be a curvilinear polygon with $C^1$-smooth edges and without cusps. If one of the following two conditions is satisfied:
	\begin{itemize}
		\item[(a)] $\eps^2-\mu^2\notin\Big(\tfrac{4}{3+2\sqrt{2}},4\big(3+2\sqrt{2}\big)\Big)$,
		\item[(b)] $\eps^2-\mu^2\notin[2,8]$ and the interior angles $\theta_j$ of $\Sigma$ satisfy
		\[
		\tfrac{\pi}{2}\le \theta_j\le \tfrac{3\pi }{2} \text{ for all } j,
		\]
	\end{itemize} 
	then $B$ is self-adjoint.
\end{corollary}

\begin{remark}\label{rem-7}
The result of Lemma \ref{lem-2} can be rewritten as  $m(\omega)= \| C_\Sigma\|^2_{\mathrm{ess}}$, where $ \| \,\cdot\,\|_{\mathrm{ess}}$ denotes the essential norm (i.e. the distance to the ideal of compact operators).
In turn, by \eqref{spece} the constant $\kappa(\omega)$ measures, modulo compact operators, how far the Cauchy projection $\tfrac{1}{2}I - C_\Sigma$ is from being an orthogonal projection.
\end{remark}

\section*{Acknowledgments}
BB and KP were supported by the Deutsche Forschungsgemeinschaft (DFG, German Research Foundation) -- 491606144. MZ was partially supported by the Universidad del Pais Vasco/EHU \& BCAM (Spain) project PES 22/50 and by the UBGRS international mobility grant  (France). This work was prepared during a visit of MZ to the Carl von Ossietzky Universit\"at Oldenburg in May--June 2023, and he would like to thank the Institut f\"ur Mathematik for the warm hospitality.

\smallskip

The authors found the mistake in the published version while working on a similar problem, and  they first corrected it using an alternative technically demanding approach based on the Mellin transform in the spirit of~\cite{rab}. The above construction based on the use of~\cite{bolt} was deduced through an interaction with ChatGPT in the course of an AI-assisted proofreading of a preliminary version of the erratum.


\end{document}